\documentclass [10pt,twoside,a4paper]{article}
\usepackage{amsfonts,amssymb,amsmath,amscd,latexsym,theorem}
\author{Luca Martinazzi\thanks{The first author was partially supported by the ETH Research Grant no. ETH-02 08-2 and by the Italian FIRB Ideas ``Analysis and Beyond".}\\ \small{Centro De Giorgi, Pisa} \\ \footnotesize{\texttt{luca.martinazzi@sns.it}}
\and Mircea Petrache\\ \small{ETH Zurich} \\ \footnotesize{\texttt{mircea.petrache@math.ethz.ch}}}
\title{Existence of solutions to a higher dimensional mean-field equation on manifolds}
\date{January 26, 2010}

\newtheorem{trm}{Theorem}

\newtheorem{lemma}[trm]{Lemma}

\theorembodyfont{\upshape}

\theorembodyfont{\upshape}

\newcommand{\R}[1]{\mathbb{R}^{#1}}

\newcommand{\de}{\partial}

\newcommand{\ve}{\varepsilon}

\newenvironment{proof}{\noindent\emph{Proof.}}{\hfill$\square$\medskip}

\DeclareMathOperator{\vol}{vol}

\begin{document}
\maketitle

\begin{abstract}
For $m\geq 1$ we prove an existence result for the equation
$$(-\Delta_g)^m u+\lambda=\lambda\frac{e^{2mu}}{\int_M e^{2mu}d\mu_g}$$
on a closed Riemannian manifold $(M,g)$ of dimension $2m$ for certain values of $\lambda$.
\end{abstract}

\section{Introduction and statement of the main result}

Let $T^2 \simeq S^1\times S^1$ be the $2$-dimensional flat torus of volume one. Motivated by the study of vortices in the Chern-Simons Gauge theory, M. Struwe and G. Tarantello \cite{ST} showed that for $\lambda\in ]4\pi,2\pi^2[$, the following equation admits a non-trivial solution\footnote{Actually \cite{ST} deals with the equation $-\Delta u+\lambda=\lambda\frac{e^{u}}{\int_{T^2}e^{u}dx}$, but upon defining $\tilde u:=2u$, $\tilde \lambda=2\lambda$ one can pass from one equation to the other.}
\begin{equation}\label{eq1}
-\Delta u+\lambda=\lambda\frac{e^{2u}}{\int_{T^2}e^{2u}dx} \quad\textrm{on }T^2.
\end{equation}

In this paper we generalize this result by considering an arbitrary closed Riemannian manifold $(M,g)$ of dimension  $2m$, and studying the equation
\begin{equation}\label{eq3}
(-\Delta_g)^m u+\lambda=\lambda\frac{e^{2mu}}{\int_Me^{2mu}d\mu_g}\quad \text{on }M,
\end{equation}
where $\Delta_g$ is the Laplace-Beltrami operator. The main theorem we shall prove is the following.
\begin{trm}\label{maintrm} Let $\lambda_1=\lambda_1(M)$ be the smallest eigenvalue of $(-\Delta_g)^m$ and $\Lambda_1:=(2m-1)!\vol(S^{2m})$. Assume that $\Lambda_1/\vol(M) <\lambda_1/(2m)$. Then for every
$\lambda\in]\Lambda_1/\vol(M) ,\lambda_1/(2m)[$, $\lambda\not\in \frac{\Lambda_1\mathbb{N}}{\vol(M)}$, \eqref{eq3} has a non-constant solution.
\end{trm}

It is easy to see that in the case when $M=T^{2m}$ is the flat torus of dimension $2m$, one has $\Lambda_1/\vol(M) <\lambda_1/(2m)$ for every $m\geq 1$, hence the theorem applies.

Notice that given a solution $u$ to \eqref{eq3}, $u+\alpha$ is also a solution for any constant $\alpha\in\R{}$, hence it is not restrictive to assume that $\int_M ud\mu_g=0$. Moreover, by a simple scaling argument we can assume that $\vol(M)=1$.

Equation \eqref{eq3} is a model for the intensively studied problems of existence and compactness properties of elliptic equations of order $4$ and higher with critical non-linearity. In fact, other than the result of Theorem \ref{maintrm} itself, also the proof is interesting, as it rests on some recent compactness results for equations arising in conformal geometry. For this reason we shall now briefly describe its strategy, which is inspired to \cite{ST}.

\medskip

Let us consider the space
$$E:=\Big\{u\in H^m(M):\int_M u d\mu_g=0\Big\},$$
with the norm
$$\|u\|:=\bigg(\int_M|\Delta_g^{\frac{m}{2}} u|^2d\mu_g\bigg)^\frac{1}{2},$$
where $\Delta^{\frac{k}{2}}_g u:=\nabla_g\Delta_g^{\frac{k-1}{2}}u$ if $k$ is odd. Then weak solutions of \eqref{eq3} are critical points of the functional
$$I_\lambda(u)=\frac{1}{2}\int_M|\Delta_g^{\frac{m}{2}} u|^2d\mu_g-\frac{\lambda}{2m} \log\bigg(\int_Me^{2mu}d\mu_g\bigg)$$
on $E$. 
By the Adams-Moser-Trudinger inequality (see \cite{adams} and Fontana \cite{fontana}), we have
\begin{equation}\label{ams}
\sup_{u\in E}\int_M e^{m\Lambda_1 \frac{u^2}{\|u\|^2}}d\mu_g<\infty, 
\end{equation}
where $\Lambda_1=(2m-1)!\vol(S^{2m})$ is the total $Q$-curvature of the round sphere of dimension $2m$, see e.g. \cite{mar1}. Then writing $2mu\leq m\Lambda_1\frac{u^2}{\|u\|^2}+\frac{m}{\Lambda_1}\|u\|^2,$ we find
\begin{equation}\label{eq4}
I_\lambda(u)\geq \bigg(\frac{1}{2}-\frac{\lambda}{2\Lambda_1}\bigg)\|u\|^2-C.
\end{equation}
Therefore $I_\lambda$ is bounded from below and coercive on $E$ for $\lambda\leq \Lambda_1$.

We shall see (Lemma \ref{poincare}) that $u\equiv 0$, which is a trivial solution to \eqref{eq3}, is a strict local minimum of $I_\lambda$ if $\lambda< \lambda_1/2m$. Moreover for $\lambda>\Lambda_1$ there always exists a function $u\in E$ such that $I_{\lambda}(u)<I_\lambda(0)=0$ (Lemma \ref{usigma}). This suggests that a mountain-pass technique might be used. In fact, as in \cite{ST}, one can use a technique of M. Struwe \cite{str} to construct a converging Palais-Smale sequence for the functional $I_\lambda$ for almost all $\lambda\in ]\Lambda_1,\lambda_1/2m[$.

In order to pass from the existence for \emph{almost every} $\lambda\in ]\Lambda_1,\lambda_1/2m[$ to the existence for \emph{all} $\lambda\in ]\Lambda_1,\lambda_1/2m[\backslash \Lambda_1\mathbb{N}$, we need a compactness argument. Given $\lambda_k$ for which a non-trivial solution $u_k$ exists, and assuming that $\lambda_k\to \lambda$, can we say that $u_k$ converges (up to a subsequence) in a good norm ($C^0$ for instance\footnote{By elliptic estimates, convergence in $C^0$ implies convergence in $C^k$ for every $k>0$.})?

In dimension $2$ this question was addressed by Brezis-Merle \cite{BM} and Li-Shafrir \cite{LS}; their result implies that if the sequence $(u_k)$ is not precompact, then $\lambda_k\to N\Lambda_1$ for some $N\in \mathbb{N}$, contrary to our assumption on $\lambda$.
As shown in \cite{ARS}, things are more subtle in higher dimension, and we cannot work locally as in \cite{ST}. Instead,  we can rely on a recent result by the first author \cite{mar2} specific for closed manifolds (see also \cite{DR}, \cite{mal} and \cite{ndi}) to obtain compactness for the sequence $(u_k)$, unless $\lambda_k\to N\Lambda_1$ for some $N\in\mathbb{N}$.

Roughly speaking, the geometric constant $\Lambda_1$ enters our problems as follows: if the sequence $(u_k)$ is not precompact, then up to a subsequence, $u_k$ concentrates at finitely many points. A blow-up argument at such points shows that the concentration profile is precisely that of a round sphere with total $Q$-curvature $\Lambda_1$.

\medskip

Related to the work of Struwe and Tarantello, several other results have been proven about the existence theory for \eqref{eq1}. For instance Ding, Jost, Li and Wang proved existence for the mean-field equation $-\Delta u=\lambda \frac{e^{2u}}{\int_\Omega e^{2u}dx}$  on an annulus $\Omega$, with boundary datum $u=0$ on $\de \Omega$ for $\lambda\in (4\pi, 8\pi)$.
Z. Djadli \cite{dja} proved the existence of solutions to \eqref{eq1} for every $\lambda \in \R{}\backslash 4\pi\mathbb{N}$. F. De Marchis \cite{dem} proved the existence of at least $2$ non-trivial solutions when $\lambda\in (4\pi, 4\pi^2)$, also in the case of a torus with nonflat metric. We refer to this last work for a more comprehensive survey of $2$-dimensional results.

These and other works usually rest on topological arguments, sometimes much more subtle than a mountain-pass principle. But a common feature is the presence of a compactness argument, which is the reason why the values $\lambda\in 4\pi\mathbb{N}$ cannot be treated. It is reasonable to believe that using the compactness result from \cite{mar2} as we did here, also these more general works can be generalized to higher dimensional manifolds and to more general semilinear equations with asymptotically exponential non-linearity. Also in this sense our Theorem \ref{maintrm} can be seen as a model situation.

\medskip

The paper is organized as follows. In Section \ref{sec2} we show that for $\lambda\in ]\Lambda_1,\lambda_1/2m[$ the constant function $u\equiv 0$ is a strict local minimum of $I_\lambda$ and that $I_\lambda$ is unbounded from below. In Section \ref{sec3} we prove the existence of a non-trivial solution to \eqref{eq3} for almost every $\lambda\in ]\Lambda_1,\lambda_1/2m[$. In Section \ref{sec4} we complete the proof of Theorem \ref{maintrm}. Finally, in Section \ref{sec5} we show that, similarly to the $2$-dimensional case, for $\lambda>0$ small enough the only solution to \eqref{eq3} is $u\equiv 0$.

\medskip

In the following, the letter $C$ denotes a generic positive constant, which may change from line to line and even within the same line.

\section{Two fundamental lemmas}\label{sec2}

We now show that for $\lambda\in ]\Lambda_1,\lambda_1/(2m)[$ the functional $I_\lambda$ is unbounded from below on $E$ and $0$ is a strict local minimum for $I_\lambda$.\\

Recall that there exists an optimal constant $C_0>0$ such that for all $v\in E$ there holds
$$
 \int_Mv^2d\mu_g\leq C_0\int_M|\nabla v|^2d\mu_g.
$$
In fact $C_0$ is the inverse of the smallest eigenvalue of $-\Delta_g$. 

\begin{lemma}\label{poincare} Let $\lambda< \frac{\lambda_1(M)}{2m}$, where $\lambda_1(M)=\frac{1}{C_0^m}$ is the smallest eigenvalue of $(-\Delta_g)^m$. Then the function $u\equiv 0$ is a strict local minimum for $I_\lambda$.
\end{lemma}

\begin{proof} Since $I_\lambda$ is smooth on $E$, it suffices to show that $I''_{\lambda}(0)$ is positive definite on $E$.
We know that $-\Delta_g$ is injective on $E$ and has an $L^2$-orthonormal  basis of eigenfunctions. Moreover, for $k>0$ and if $v_j\in E$ is the eigenfunction corresponding to the eigenvalue $\lambda_j$ of $-\Delta_g$ we have
$$
(-\Delta_g)^kv_j=(\lambda_j)^kv_j,
$$
hence $\{v_j\}$ is also an orthonormal basis of eigenfunctions for $(-\Delta_g)^m$, whose smallest eigenvalue is therefore $C_0^m$. Moreover
\begin{equation}\label{poinc}
C_0^m=\sup_{\|v\|=1}\int_M v^2d\mu_g,
\end{equation}
so that $C_0^m$ is the best constant such that for $v\in E$ there holds 
\begin{equation}\label{I''}
I''_\lambda(0)(v,v)=\|v\|^2-2m\lambda\int_M v^2d\mu_g	\geq\bigg(1-\frac{2m\lambda}{\lambda_1(M)}\bigg)\|v\|^2,
\end{equation}
and the result of the lemma easily follows.
\end{proof}

%
%

According to \eqref{eq4} $I_\lambda$ is bounded from below for $\lambda\leq \Lambda_1$. The following lemma shows that this result is sharp.

\begin{lemma}\label{usigma} There is a one-parameter family of functions $(u_\sigma)_{\sigma>0}\subset E\cap C^\infty(M)$ such that for every $\lambda>0$
\begin{eqnarray}
\|u_\sigma\|&=&(2\Lambda_1+o(1))\log\sigma,\label{eqB}\\
I_\lambda(u_\sigma)&=&(\Lambda_1-\lambda+o(1))\log\sigma\label{eqA}
\end{eqnarray}
with error $o(1)\to 0$ as $\sigma\rightarrow\infty$. In particular, if $\lambda>\Lambda_1$ then $I_\lambda$ is not bounded from below.
\end{lemma}

\begin{proof} We divide the proof into steps.

\medskip

\noindent \emph{Step 1: Construction of $u_\sigma$ and proof of \eqref{eqB}}. Let $\varphi\in C^\infty_c(B_1)$ be a radially symmetric function such that $0\leq \varphi\le 1$ on $B_1$, $\varphi\equiv 1$ on $B_{1/4}$ and $\varphi\equiv 0$ on $B_1\backslash B_{1/2}$.
Set
$$v_\sigma(x):=\varphi(x)\log\bigg(\frac{2\sigma}{1+\sigma^2|x|^2}\bigg)+(1-\varphi(x))\log\bigg(\frac{2\sigma}{1+\sigma^2}\bigg),\quad x\in B_1,$$
so that $v_\sigma\in C^\infty(B_1)$ and, since $r\mapsto\log(2\sigma/(1+\sigma^2r^2))$ is decreasing and $\varphi\geq 0$, there holds
\begin{equation}\label{vs}
\log\bigg(\frac{2\sigma}{1+\sigma^2}\bigg)\le v_\sigma(x)\leq \log\bigg(\frac{2\sigma}{1+\sigma^2|x|^2}\bigg),\quad x\in B_1.
\end{equation}
Set $w_\sigma(r):=\log(\tfrac{2\sigma}{1+\sigma^2r^2})$ and (with an abuse of notation) write $\varphi(x)=\varphi(r)$, with $r:=|x|$, so that 
$$
v_\sigma(x)=\varphi(r)w_\sigma(r) + (1-\varphi(r))w_\sigma(1).
$$
For two radial functions $f(r),g(r)$ we have
\begin{equation}\label{radiallaplacian}
 \Delta(fg)=f\Delta g+g\Delta f+2f'g',\quad \nabla(fg)=f\nabla g+g\nabla f,
\end{equation}
and
$$
\Delta f=f''+\frac{2m-1}{r}f',\quad  \nabla f(x)=\frac{x}{|x|} f'(|x|),
$$
hence, up to identifying $\nabla f(x)$ with $f'(|x|)$,
we may repeatedly use \eqref{radiallaplacian} to get
\begin{equation}\label{lapv}
\Delta^{\frac{m}{2}}v_\sigma=\varphi\Delta^{\frac{m}{2}}w_\sigma+\sum_{\substack{j+k+\ell=m\\j,\ell\geq 0,\;k\geq 1}}C_{jk\ell m}\frac{\de_r^j(w_\sigma-w_\sigma(1)) \de_r^k\varphi}{r^\ell},
\end{equation}
for some dimensional constants $C_{jk\ell m}$. Observe that $\de_r^k\varphi$ is supported in $B_{1/2}\setminus B_{1/4}$ for $k\geq 1$, and $\|\de_r^k\varphi\|_{L^\infty}\leq C(k)$ for every $k\geq 0$.
We now claim that
\begin{equation}\label{Wj}
|\de_r^j(w_\sigma(r)-w_\sigma(1))|=O(r^{-j})\quad \text{for } j\geq 0,\; \frac{1}{4}\le r\le \frac{1}{2},
\end{equation}
as $\sigma\to \infty$.
Indeed for $j=0$ and $\frac{1}{4}\le r\le \frac{1}{2}$ we have
$$
|w_\sigma-w_\sigma(1)|=\left|w_\sigma-\log\bigg(\frac{2\sigma}{1+\sigma^2}\bigg)\right|\leq C\left|\log\bigg(\frac{1+\sigma^2}{1+(\sigma/4)^2}\bigg)\right|=O(1),$$
as $\sigma\rightarrow\infty$, and for $j=1$
$$|\de_r(w_\sigma-w_\sigma(1))|=|w'_\sigma|=\left|\frac{2\sigma^2r}{1+\sigma^2r^2}\right|=O(r^{-1})\quad  \text{as } \sigma\rightarrow\infty.$$
For $ j\geq 2$, observe that $\de_r^j w_\sigma=\sigma^j\frac{P_j(\sigma r)}{Q_j(\sigma r)}$ for some polynomials $P_j$ and $Q_j$. In fact we have
$$
\frac{d}{dr}\frac{P_j(\sigma r)}{Q_j(\sigma r)}=\sigma\frac{P_j'(\sigma r)Q_j(\sigma r)-Q_j'(\sigma r)P_j(\sigma r)}{Q_j^2(\sigma r)}=:\frac{P_{j+1}(\sigma r)}{Q_{j+1}(\sigma r)}.
$$
Then clearly
$$\bigg|\frac{P_{j+1}(\sigma r)}{Q_{j+1}(\sigma r)} \bigg| \leq \frac{C}{r}\bigg|\frac{P_j(\sigma r)}{Q_j(\sigma r)}\bigg|\quad \text{as } \sigma\to\infty,$$
and \eqref{Wj} follows by induction.

Since in the sum in \eqref{lapv} there is no term with more than $m-1$ derivatives of $w_\sigma$, and by the bounds on $\varphi$, we then have for $\sigma$ large

\begin{equation}\label{loworder}
\begin{split}
\int_{B_1}|\Delta^{\frac{m}{2}} v_\sigma-\varphi\Delta^{\frac{m}{2}} w_\sigma|^2dx&\leq C\sum_{0\leq j+\ell\leq m+1}\int_{B_{1/2}\setminus B_{1/4}}\frac{|\de_r^j (w_\sigma-w_\sigma(1))|^2}{r^{2\ell}}dx\\
& \le C \sum_{0\leq j+\ell\leq m+1}\int_{B_{1/2}\setminus B_{1/4}} r^{-2j-2\ell} dr \\
& \le C\int_{\frac{1}{4}}^{\frac{1}{2}}rdr=C.
\end{split}
\end{equation}
Also $\sigma^{-m}\Delta^{\frac{m}{2}} w_\sigma$ is the quotient of two polynomials in $\sigma r$. In fact 
\begin{equation}\label{highorder}
|\Delta^{\frac{m}{2}} w_\sigma|=2^m(m-1)!\sigma^m\frac{\sigma^{m}r^m+p(\sigma r)}{(1+\sigma^2r^2)^m},
\end{equation}
where $\deg p\leq m-1$.
Then \eqref{loworder}, \eqref{highorder} and the change of variable $s=1+\sigma^2r^2$ yield
\begin{equation}\label{allorders}
\begin{split}
 \int_{B_1}|\Delta^{\frac{m}{2}} v_\sigma|^2dx&=\omega_{2m-1}(2^m(m-1)!)^2\int_0^1\frac{\sigma^{4m}r^{4m-1}}{(1+\sigma^2r^2)^{2m}}dr+O(1)\\
&= 2\Lambda_1\int_1^{1+\sigma^2}\frac{(s-1)^{2m-1}}{2s^{2m}}ds+O(1)\\
&= 2\Lambda_1\log\sigma +O(1),
\end{split}
\end{equation}
with error $|O(1)|\le C$ as $\sigma\to \infty$.

Fix now $p\in M$ and take $\alpha>0$ smaller than the injectivity radius of $(M,g)$. Consider the map $f_\alpha:B_1\rightarrow M$ given by $f_\alpha(x):=\exp_p(\alpha x)$, where $\exp_p$ is the exponential map at $p$.
Then we define
$$
\tilde v_{\sigma,\alpha}:=\left\{\begin{array}{ll}
                         v_\sigma\circ f_\alpha^{-1}&\text{on }K_\alpha:=f_\alpha(B_1)\\
                         \log\frac{2\sigma}{1+\sigma^2}&\text{on }M\setminus K_\alpha,
                       \end{array}\right.
$$
and
$$
u_{\sigma,\alpha}:=\tilde v_{\sigma,\alpha} - \int_M \tilde v_{\sigma,\alpha} d\mu_g\in E.
$$
We also consider the metric $h_\alpha:=\alpha^{-2}f^*_\alpha g$ on $B_1$.
We claim that
\begin{equation}\label{eqC}
 \int_M|\Delta^{\frac{m}{2}}_g u_{\sigma,\alpha}|_g^2d\mu_g= \int_{K_\alpha}|\Delta^{\frac{m}{2}}_g( v_\sigma\circ f^{-1}_\alpha)|_g^2d\mu_g=\int_{B_1}|\Delta^{\frac{m}{2}}_{h_\alpha} v_{\sigma}|_{h_\alpha}^2d\mu_{h_\alpha}.
\end{equation}
The first identity in \eqref{eqC} is clear. In order to prove the second one, consider first the case when $m$ is even. Then, writing
\begin{equation*}
\begin{split}
h_{\alpha,ij}&:=h_\alpha \Big(\frac{\de}{\de x^i}, \frac{\de}{\de x^j}\Big)=\alpha^{-2}g\Big(\frac{\de f_\alpha}{\de x^i}, \frac{\de f_\alpha}{\de x^j}\Big)=:\alpha^{-2}g_{ij}\\
\sqrt{h_\alpha}&:=\sqrt{\det(h_{\alpha,ij})}=\alpha^{-2m}\sqrt{\det (g_{ij})}=:\alpha^{-2m}\sqrt{g},
\end{split}
\end{equation*}
and using the summation convention, we compute
\begin{equation*}
\begin{split}
\int_{K_\alpha}&\big(\Delta^{\frac{m}{2}}_g (v_\sigma\circ f_\alpha^{-1})\big)^2d\mu_g\\
&=\int_{B_1}\bigg\{\bigg(\frac{1}{\sqrt{g}} \frac{\de}{\de x^i}\Big(\sqrt{g}g^{ij}\frac{\de}{\de x^j}\Big)  \bigg)^{m/2}v_\sigma\circ f_\alpha^{-1}(f_\alpha(x))  \bigg\}^2\sqrt{g}dx\\
&= \int_{B_1}\bigg\{\bigg(\frac{\alpha^{-2}}{\sqrt{h_\alpha}} \frac{\de}{\de x^i}\Big(\sqrt{h_\alpha}h_\alpha^{ij}\frac{\de}{\de x^j}\Big)  \bigg)^{m/2}v_\sigma(x) \bigg\}^2\alpha^{2m}\sqrt{h_\alpha}dx\\
&=\int_{B_1}(\Delta_{h_\alpha}^\frac{m}{2} v_\sigma)^2 d\mu_{h_\alpha}.
\end{split}
\end{equation*}
This proves \eqref{eqC} for $m$ even. When $m$ is odd the argument is similar, additionally using the formula
$$\int_{K_\alpha}|\nabla \psi|_g^2d\mu_g=\int_{B_1}g^{ij}\frac{\de}{\de x^i}\psi(f_\alpha(x)) \frac{\de}{\de x^j}\psi(f_\alpha(x))\sqrt{g}dx$$
for any $\psi\in C^\infty(K_\alpha)$.
Since the metric $g$ is smooth, we have that $h_\alpha\to |dx|^2$ as $\alpha\to0$ in $C^\ell(B_1)$ for every $\ell$, where $|dx|^2$ denotes the Euclidean metric. In particular, using \eqref{allorders}, we see that there exists a function $\ve=\ve(\alpha)$ defined for $\alpha$ small with $\lim_{\alpha \to 0^+}\ve (\alpha)=0$ such that
\begin{equation}\label{perturb}
(1-\ve(\alpha))2\Lambda_1\log\sigma -C\leq\int_{B_1}|\Delta^{\frac{m}{2}}_{h_{\alpha}} v_\sigma|_{h_{\alpha}}^2d\mu_{h_{\alpha}}\leq (1+\ve(\alpha))2\Lambda_1\log\sigma+C.
\end{equation}
For each $\sigma>1$ choose $\alpha=\alpha(\sigma)$ such that
\begin{equation}\label{as}
\lim_{\sigma\to\infty}\alpha(\sigma)=0,\quad \lim_{\sigma\to\infty}\sigma \alpha(\sigma)=\infty.
\end{equation}
Then setting $u_\sigma:=u_{\sigma,\alpha(\sigma)}$ and taking into account \eqref{eqC}, \eqref{perturb} and \eqref{as}, we infer \eqref{eqB}.

\medskip

\noindent \emph{Step 2: Proof of \eqref{eqA}}. It remains to estimate 
$$\frac{1}{2m}\log\left(\int_M e^{2mu_\sigma}d\mu_g\right)=\frac{1}{2m}\log\left(\int_M e^{2m\tilde v_{\sigma,\alpha}}d\mu_g\right)-\int_M\tilde v_{\sigma,\alpha} d\mu_g=:I-II.
$$
We claim that
\begin{eqnarray}
I&=&\log\alpha+O(1),\label{stima2pezzo1} \\
II&=& -(1+o(1))\log\sigma,\label{stima2pezzo2}
\end{eqnarray}
with errors $|O(1)|\leq C$ and $o(1)\to 0$ as $\sigma\to\infty$.
As for \eqref{stima2pezzo2} we have
\begin{equation*}
II=\int_{K_\alpha}v_\sigma\circ f^{-1}_\alpha d\mu_g+\int_{M\setminus K_\alpha} \log\bigg(\frac{2\sigma}{1+\sigma^2}\bigg)d\mu_g=:III+IV
\end{equation*}
Since $\vol (M\backslash K_\alpha)\to 1$ as $\sigma\to \infty$, we have
$$IV=-(1+o(1))\log \sigma,$$
with error $o(1)\to 0$ as $\sigma\to\infty$.
Defining $h_\alpha$, $\sqrt{g}$ and $\sqrt{h_\alpha}$  as above, with $\alpha=\alpha(\sigma)$, using that $h_{\alpha(\sigma)}\to |dx|^2$ as $\sigma\to\infty$ in $C^\ell(B_1)$ for every $\ell\geq 0$, \eqref{vs} and \eqref{as}, we also get
\begin{equation*}
\begin{split}
III=\int_{B_1}v_\sigma \sqrt{g}dx&=\alpha^{2m}\int_{B_1}v_\sigma\sqrt{h_\alpha}dx=\alpha^{2m}(1+o(1))\int_{B_1}v_\sigma dx\\
&=\alpha^{2m}(1+o(1))(-\log\sigma)=o(1)\log\sigma,
\end{split}
\end{equation*}
with error $o(1)\to 0$ as $\sigma\to\infty$.
Therefore \eqref{stima2pezzo2} is proved.

We shall now prove \eqref{stima2pezzo1}. We have
$$
A:=\int_M e^{2m\tilde v_{\sigma,\alpha}}d\mu_g= \int_{M\setminus K_\alpha}\left(\frac{2\sigma}{1+\sigma^2}\right)^{2m}d\mu_g+\int_{K_\alpha}e^{2m\tilde v_{\sigma,\alpha}}d\mu_g=:A_1+A_2.
$$
We clearly have $A_1\to 0$ as $\sigma\to\infty$, and
$$A_2=\alpha^{2m}\int_{B_1}e^{2mv_\sigma} d\mu_{h_\alpha}=\alpha^{2m}(1+o(1))\int_{B_1}e^{2mv_\sigma}dx.$$
Therefore
$$I=\frac{1}{2m}\log A=\log\alpha+\frac{1}{2m}\log\bigg(\int_{B_1}e^{2mv_\sigma}dx\bigg)+o(1),$$
with error $o(1)\to 0$ as $\sigma\to\infty$ and we complete the proof of \eqref{stima2pezzo1} by showing that
\begin{equation}\label{ultima}
\frac{1}{C}\leq \int_{B_1}e^{2mv_\sigma}dx\leq C.
\end{equation}
Observe that for $|x|\le 1$ we have
\begin{equation*}
\begin{split}
\chi_{B_1\setminus B_\frac{1}{4}}\log\bigg(\frac{2\sigma}{1+\sigma^2}\bigg)+&\chi_{B_{\frac{1}{4}}}(x)\log\bigg(\frac{2\sigma}{1+\sigma^2|x|^2}\bigg)\\
&\leq v_\sigma(x) 
\leq \log\bigg(\frac{2\sigma}{1+\sigma^2|x|^2}\bigg),
\end{split}
\end{equation*}
hence
$$ \int_{B_{1/4}}\bigg(\frac{2\sigma}{1+\sigma^2|x|^2}\bigg)^{2m}dx \leq \int_{B_1}e^{2mv_\sigma}dx\leq \int_{B_1}\bigg(\frac{2\sigma}{1+\sigma^2|x|^2}\bigg)^{2m}dx.$$
Now \eqref{ultima} follows observing that for any fixed $R>0$ one has
$$\int_{B_R}\bigg(\frac{2\sigma}{1+\sigma^2|x|^2}\bigg)^{2m}dx=\int_{B_{\sigma R}}\bigg(\frac{2}{1+|y|^2}\bigg)^{2m}dy=C_0+o(1),$$
with error $o(1)\to 0$ as $\sigma\to\infty$, where $C_0=\int_{\R{2m}}\big(2/(1+|y|^2)\big)^{2m}dy<\infty$.

Together with Step 1, we have shown that
$$I_\lambda(u_\sigma)=(\Lambda_1-\lambda+o(1))\log\sigma -\lambda\log(\alpha),\quad \text{as }\quad \sigma\to\infty.$$
Observing that \eqref{as} implies $\log\alpha=o(1)\log\sigma$ as $\sigma\to\infty$, we infer \eqref{eqA}.
\end{proof}

\section{Existence for almost every $\lambda\in ]\Lambda_1,\lambda_1/2m[$} \label{sec3}

Fix $\lambda\in]\Lambda_1,\lambda_1/2m[$. By Lemma \ref{usigma}, there exists $\sigma=\sigma(\lambda)>0$ such that for $u_0:=u_\sigma$ we have
$$I(u_0)<0\quad \textrm{and} \quad \|u_0\|\geq 1.$$
Consider the set of paths
$$P:=\{\gamma\in C^0([0,1];E): \gamma(0)=0, \gamma(1)=u_0, \gamma(t)\in C^\infty(M) \text{ for }0\le t\le 1\},$$
which is clearly non-empty since $u_0\in E\cap C^\infty(M)$, and for $\mu\in]\lambda,\lambda_1/2m[$ set
$$c_\mu:=\inf_{\gamma\in P} \max_{t\in[0,1]} I_\mu(\gamma(t)).$$
Since by Jensen's inequality $\log(\int_Me^{2mu}d\mu_g)>0$, the function $\mu\mapsto c_\mu$ is non-increasing, hence differentiable for almost every $\mu\in]\lambda,\lambda_1/2m[$. Then we will show that for any $\mu$ such that $c'_\mu:=dc_\mu/d\mu$ exists, the functional $I_\mu$ admits a converging Palais-Smale sequence at level $c_\mu$.

\begin{lemma}\label{derivI} \begin{enumerate}\item For any $u,v\in E$, $\mu\geq 0$ there holds
$$I_\mu(u+v)\leq I_\mu(u)+\langle I'_\mu(u),v\rangle+\frac{1}{2}\|v\|^2,$$
where
$$\langle I'_\mu(u),v\rangle:=\frac{d}{dt}I_\mu(u+tv)\Big|_{t=0}.$$
\item For any $C_1\geq 0$ there exists a constant $\tilde C_1$ such that for any $\mu,\nu\in\R{}$ there holds
$$\|I'_\mu(u)-I'_\nu(u)\|\leq \tilde C_1|\mu-\nu|,$$
uniformly in $u\in E$ with $\|u\|^2\leq C_1$, where
$$\|I'_\mu(u)\|:=\sup_{\|v\|\le 1}\langle I'_\mu(u),v\rangle.$$
\end{enumerate}
\end{lemma}

\begin{proof} \emph{1.} We have
\begin{multline*}
I_\mu(u+v)-I_\mu(u)-\langle I'_\mu(u),v\rangle -\frac{1}{2}\|v\|^2\\
=-\frac{\mu}{2m}\log\bigg(\frac{\int_Me^{2m(u+v)}d\mu_g}{\int_Me^{2mu}d\mu_g}\bigg)+\mu\frac{\int_Mve^{2mv}d\mu_g}{\int_Me^{2mu}d\mu_g}=-\frac{\mu}{2m}\int_0^1\int_0^{t}f''(s)dsdt,
\end{multline*}
where $f(s)=\log\big(\int_Me^{2m(u+sv)}d\mu_g\big/\int_Me^{2mu}d\mu_g\big)$.
By H\"older's inequality
\begin{eqnarray*}
f''(s)&=&\bigg[4m^2\int_M v^2 e^{2m(u+sv)}d\mu_g \int_M e^{2m(u+sv)} d\mu_g\\
&&-\bigg( 2m \int_M ve^{2m(u+sv)} d\mu_g \bigg) ^2 \bigg] \times \bigg(\int_Me^{2m(u+sv)}d\mu_g\bigg)^{-2}\\
&\geq& 0.
\end{eqnarray*}

\noindent\emph{2.} Take $u,v\in E$ with $\|v\|\leq 1$. Recalling that $\int_Me^{2mu}\geq 1$ and using \eqref{poinc}, we get
\begin{eqnarray*}
\langle I'_\mu(u),v\rangle-\langle I'_\nu(u),v\rangle&=&(\nu-\mu)\frac{\int_Mve^{2mu}d\mu_g}{\int_Me^{2mu}d\mu_g}\\
&\leq&|\mu-\nu|\bigg(\int_Me^{4mu}d\mu_g \int_Mv^2d\mu_g\bigg)^\frac{1}{2}\\
&\leq&C_P^{\frac{m}{2}}|\mu-\nu|\bigg(\int_Me^{4mu}d\mu_g\bigg)^\frac{1}{2}.
\end{eqnarray*}
Applying Fontana's inequality \eqref{ams} together with $4mu\leq m\Lambda_1\frac{u^2}{\|u\|^2}+\frac{4m}{\Lambda_1}\|u\|^2$, and recalling that $\|u\|\leq C_1$
we find
$$\bigg(\int_Me^{4mu}d\mu_g\bigg)^\frac{1}{2}\leq C\bigg(\int_Me^{m\Lambda_1\frac{u^2}{\|u\|^2}}d\mu_g\bigg)^\frac{1}{2}\leq C,$$
and we conclude.
\end{proof}

\begin{lemma}\label{PS}
 Fix $\mu\in ]\lambda,\lambda_1/(2m)[$ such that the derivative $c'_\mu$ exists. Then there exists a sequence $(u_n)\subset E\cap C^\infty(M)$ such that $\|u_n\|^2\leq C$, $I_\mu(u_n)\to c_\mu$ and $I'_\mu(u_n)\rightarrow 0$.
\end{lemma}

\begin{proof}
Suppose that the lemma is false. Then for each $C_0>0$ there exists $\delta(C_0)>0$ for which $\|u\|^2\leq C_0$ and $|I_\mu(u)-c_\mu|<2\delta$ imply $\|I'_\mu(u)\|\geq 2\delta$. We set $\alpha:=-c'_\mu+3\geq 3$, we consider a decreasing sequence $\mu_n\rightarrow \mu$ and a sequence of paths $\gamma_n\in P$ such that 
\begin{equation}\label{eq5}
\max_{0\leq t\leq 1} I_\mu(\gamma_n(t))\leq c_\mu+(\mu_n-\mu).
\end{equation}
Take $v_n=\gamma_n(t_n)$ such that 
\begin{equation}\label{imun}
 I_{\mu_n}(v_n)\geq c_{\mu_n}-2(\mu_n-\mu).
\end{equation}
Then for $n$ sufficiently large
\begin{equation}\label{eq6}
\begin{split}
c_\mu-\alpha(\mu_n-\mu)&\leq c_{\mu_n}-2(\mu_n-\mu)\leq I_{\mu_n}(v_n)\leq I_\mu(v_n)\\
&\leq\max_{t\in [0,1]} I_\mu(\gamma_n(t)) \leq c_\mu+(\mu_n-\mu).
\end{split}
\end{equation}
In particular
$$I_\mu(v_n)-I_{\mu_n}(v_n)\leq c_\mu+(\mu_n-\mu)-(c_\mu-\alpha(\mu_n-\mu))=(\alpha+1)(\mu_n-\mu), $$
so that 
$$\frac{I_{\mu}(v_n)-I_{\mu_n}(v_n)}{\mu_n-\mu}=\frac{1}{2m}\log\left(\int_Me^{2mv_n}d\mu_g\right)\leq\alpha+1.$$
This and \eqref{eq5} yield
\begin{equation}\label{eq7} 
\|v_n\|^2=2I_\mu(v_n)+\frac{\mu}{m}\log\left(\int_Me^{2mv_n}d\mu_g\right)\leq C(\mu)=:C_1.
\end{equation}
By assumption we can now choose $\delta=\delta(C_1)$ so that for $n$ sufficiently large if $|I_\mu(v_n)-c_\mu|<2\delta$, then $\|I'_\mu(v_n)\|\geq 2\delta$. By Lemma \ref{derivI} we get
\begin{eqnarray}\label{scalarpr}
\begin{split}
\langle I'_{\mu_n}(v_n),I'_\mu(v_n) \rangle&=\|I'_\mu(v_n)\|^2-\langle I'_\mu(v_n)-I'_{\mu_n}(v_n), I'_\mu(v_n)\rangle\\
&\geq \frac{1}{2}\|I'_\mu(v_n)\|^2-\frac{1}{2}\|I'_\mu(v_n)-I'_{\mu_n}(v_n) \|^2\\
&\geq \frac{1}{2}\|I'_\mu(v_n)\|^2 -\tilde C_1|\mu-\mu_n|^2\\
&\geq \frac{1}{4}\|I'_\mu(v_n)\|^2 \geq\delta^2,
\end{split}
\end{eqnarray}
for $n$ sufficiently large.
Now choose $\varphi\in C^\infty(\R{})$ such that $0\leq \varphi\leq 1$ with $\varphi\equiv 1$ on $[-1,\infty)$ and $\varphi\equiv 0$ on $(-\infty,-2]$. For $n\in\mathbb{N}$ and $u\in E$ set
$$\varphi_n(u):=\varphi\bigg(\frac{I_{\mu_n}(u)-c_\mu}{\mu_n-\mu}\bigg).$$
With $\gamma_n\in P$ and $v_n=\gamma_n(t_n)$ as above we set
$$\tilde \gamma_n(t):= \gamma_n(t)-\sqrt{\mu_n-\mu}\;\varphi_n(\gamma_n(t))\frac{I'_\mu(\gamma_n(t))}{\|I'_\mu(\gamma_n(t))\|}\in E\cap C^\infty(M), $$
and $\tilde v_n=\tilde \gamma_n(t_n)$. Then we get from Lemma \ref{derivI} and \eqref{scalarpr}
\begin{equation}\label{imutilde}
\begin{split}
I_{\mu_n}(\tilde v_n)&=I_{\mu_n}\bigg(v_n-\sqrt{\mu_n-\mu}\;\varphi_n(v_n)\frac{I_\mu'(v_n)}{\| I_\mu'(v_n)\|}\bigg) \\
&\leq I_{\mu_n}(v_n)-\frac{\sqrt{\mu_n-\mu}\;\varphi_n(v_n)}{\|I_\mu'(v_n)\|}\langle I_{\mu_n}'(v_n),I_{\mu}'(v_n)\rangle +\frac{1}{2}(\mu_n-\mu)\varphi_n^2(v_n)\\
&\leq I_{\mu_n}(v_n)-\frac{1}{4}\sqrt{\mu_n-\mu}\;\varphi_n(v_n)\|I'_\mu(v_n)\|+\frac{1}{2}(\mu_n-\mu)\varphi^2_n(v_n)\\
&\leq I_{\mu_n}(v_n)-\frac{\delta}{4}\sqrt{\mu_n-\mu}\;\varphi_n(v_n)\leq I_{\mu_n}(v_n),
\end{split}
\end{equation}
for $n$ large enough.
Now we claim that for $n$ large enough
\begin{equation}\label{cmu}
c_{\mu_n}\leq \max_{0\leq t\leq 1}I_{\mu_n}(\tilde \gamma_n(t))=\max_{\{t\in[0,1]: I_{\mu_n}(\gamma_n(t))\geq c_{\mu_n}-(\mu_n-\mu)\}} I_{\mu_n}(\tilde \gamma_n(t)).
\end{equation}
The inequality is clear. As for the identity, observe that if $t\in [0,1]$ is such that $I_{\mu_n}(\gamma_n(t))\le c_{\mu_n}-2(\mu_n-\mu)$, then $\tilde \gamma_n(t)=\gamma_n(t)$, hence 
$$I_{\mu_n}(\tilde \gamma_n(t))=I_{\mu_n}(\gamma_n(t))<c_{\mu_n}.$$
If $t\in [0,1]$ is such that
$$I_{\mu_n}(\gamma_n(t))\in ]c_{\mu_n}-2(\mu_n-\mu),c_{\mu_n}-(\mu_n-\mu)[,$$
then \eqref{imun} holds for $v_n=\gamma_n(t)$ and we can apply \eqref{imutilde} with $\tilde v_n=\tilde \gamma_n(t)$ and infer
$$I_{\mu_n}(\tilde \gamma_n(t))\leq I_{\mu_n}(\gamma_n(t))<c_{\mu_n}.$$
Then \eqref{cmu} is proven and, since for $t$ such that $I_{\mu_n}(\gamma_n(t))\geq c_{\mu_n}-(\mu_n-\mu)$ we have that \eqref{imutilde} holds for $v_n=\gamma_n(t)$ and $\tilde v_n=\tilde \gamma_n(t)$ with $\varphi(\gamma_n(t))=1$, recalling \eqref{eq5} and \eqref{eq6}, we infer
\begin{eqnarray*}
c_{\mu_n}&\leq&\max_{\{t\in[0,1]: I_{\mu_n}(\gamma_n(t))\geq c_{\mu_n}-(\mu_n-\mu)\}} I_{\mu_n}(\tilde \gamma_n(t))\\
&\leq&\max_{0\leq t\leq 1} I_{\mu_n}(\gamma_n(t))-\frac{\delta}{4}\sqrt{\mu_n-\mu}\leq \max_{0\leq t\leq 1} I_{\mu}(\gamma_n(t))-\frac{\delta}{4}\sqrt{\mu_n-\mu}\\
&\leq& c_\mu +(\mu_n-\mu)-\frac{\delta}{4}\sqrt{\mu_n-\mu}\leq c_{\mu_n}+(\alpha-1)(\mu_n-\mu)-\frac{\delta}{4}\sqrt{\mu_n-\mu}\\
&<&c_{\mu_n},
\end{eqnarray*}
for $n$ large enough, contradiction.
\end{proof}

\begin{lemma}\label{exist1}
 If $\mu\mapsto c_\mu$ is differentiable at $\mu$ then $c_\mu$ is a critical value of $I_\mu$.
\end{lemma}
\begin{proof}
By Lemma \ref{PS} there exists a bounded sequence $(u_n)$ in $E$ such that $I'_\mu(u_n)\rightarrow 0$ and $I_\mu(u_n)\to c_\mu$. We may assume that $u_n$ converges weakly in $E$ and almost everywhere to a function $u$. Moreover we can use Fontana's inequality together with the inequality $8mu\leq m\Lambda_1\tfrac{u^2}{||u||^2} +\tfrac{16m}{\Lambda_1}||u||^2$ as in the proof of Lemma \ref{derivI} to show that $e^{2mu_n}$ and $e^{2mu}$ are uniformly bounded in $L^4$. Observing that by dominated convergence one has for $N>0$
$$\min\{e^{2mu_n},N\}\to \min\{e^{2mu},N\}\quad\text{in }L^2(M,d\mu_g)$$
as $n\to\infty$ and that
$$\sup_{n\in\mathbb{N}}\|\min\{e^{2mu_n},N\} -e^{2mu_n}\|_{L^2}^2\le \frac{1}{N^2}\sup_{n\in\mathbb{N}}\|e^{2mu_n}\|_{L^4}^4\to 0\quad \text{as }N\to \infty,$$
we infer that $e^{2mu_n}\rightarrow e^{2mu}$ in $L^2$. Then we have
$$
o(1)=\langle I'_\mu(u_n),u_n-u\rangle=\|u_n-u\|^2+o(1),
$$
with error $o(1)\to 0$ as $n\to\infty$. This proves that $u_n\to u$ in $E$, hence $u$ is a critical point of $I_\mu$ with $I_\mu(u)=c_\mu$.
\end{proof}

\section{Compactness and proof of Theorem \ref{maintrm}}\label{sec4}

The following theorem follows from \cite[Thm. 2]{mar2}, compare also \cite{BM}, \cite{LS}, \cite{DR}, \cite{mal} and \cite{MP}.

\begin{trm}\label{trmMP1}
Let $u_k\in C^\infty(M)$ be a sequence of solutions to
\begin{equation}\label{paneitz}
 (-\Delta_g)^m u_k + \lambda_k=\lambda_k \frac{e^{2mu_k}}{\int_Me^{2mu_k}d\mu_g},
\end{equation}
where $\lambda_k\rightarrow \lambda$ are positive real numbers. Then one of the following is true:
\begin{itemize}
\item[(i)] Up to a subsequence $u_k\to u_0$ in $C^{2m-1}(M)$ for some $u_0\in C^\infty(M)$.
\item[(ii)] Up to a subsequence, $\lim_{k\to\infty}\max_M u_k=\infty$ and there is a positive integer $N$ such that
\begin{equation}\label{quant}
\lim_{k\to\infty}\lambda_k=N\Lambda_1.
\end{equation}
\end{itemize}
\end{trm}

\begin{proof} In \cite{mar2} the equation
$$P_{g}^{2m}u_k+Q_g=Q_ke^{2mu_k}$$
is treated, where $P_g^{2m}$ is the Paneitz (or GJMS) operator of the Riemannian manifold $(M,g)$, $Q_g\in C^\infty(M)$ (it is the $Q$-curvature of $(M,g)$) and $Q_k\to Q_0$ in $C^1(M)$ is a given sequence. Under these assumptions it is proven that up to a subsequence either
\begin{itemize}
\item[(i)] $u_k\to u_0$ in $C^{2m-1}(M)$ for some $u_0\in C^\infty(M)$, or
\item[(ii)] $\lim_{k\to\infty}\max_M u_k=\infty$ and there is a positive integer $N$ such that
\begin{equation}\label{quant0}
\lim_{k\to\infty}\int_M Q_k e^{2mu_k}d\mu_g=N\Lambda_1.
\end{equation}
\end{itemize}
But in fact the proof of \cite{mar2} applies to more general equations of the form
\begin{equation}\label{general}
L_g u_k + f_k=h_k e^{2mu_k},
\end{equation}
where
\begin{enumerate}
\item $L_g$ is any differential operator of the form
$L_g=(-\Delta_g)^m+A_k$, where $A_k$ is a differential operator of order $2m-1$ at most and whose coefficients converge in $C^1$;
\item $f_k\to f_0$ in $C^1$ and $h_k\to h_0$ in $C^1$,
\end{enumerate}
see e.g. \cite{DR}. In this case the conclusion is that if $(u_k)$ is not precompact in $C^{2m-1}(M)$, then up to a subsequence
\begin{equation}\label{quant2}
\lim_{k\to\infty}\int_{M}h_k e^{2mu_k}d\mu_g=N\Lambda_1
\end{equation}
for some $N\in \mathbb{N}$.

Solutions to \eqref{paneitz} are also solutions to \eqref{general} with $f_k\equiv \lambda_k\to\lambda$ and
$$h_k\equiv\frac{\lambda_k}{\int_Me^{2mu_k}d\mu_g}\to h_0\equiv const\ge 0\quad \text{as }k\to\infty,$$
up to a subsequence. If the sequence $(u_k)$ is not precompact in $C^{2m-1}(M)$, then \eqref{quant2} implies \eqref{quant} at once.
\end{proof}

\medskip

\noindent\emph{Proof of Theorem \ref{maintrm} (completed).} For $\lambda\in]\Lambda_1,\lambda_1/2m[$, $\lambda\not\in \Lambda_1\mathbb{N}$, consider a sequence $\lambda_k<\lambda$ with $\lambda_k\to\lambda$ such that for every $k>0$ there is a solution $u_k\in E$ to \eqref{paneitz} with $I_{\lambda_k}(u_k)=c_{\lambda_k}$. That such a sequence $(\lambda_k,u_k)$ exists was shown in Lemma \ref{exist1}. Moreover Lemma \ref{poincare} implies that $c_{\lambda}>0$.
Theorem \ref{trmMP1} then implies that (up to a subsequence) $u_k\to u_\lambda$ in $C^{2m-1}(M)$, hence smoothly, for some function $u_\lambda\in C^\infty(M)$, which also satisfies \eqref{eq3}. Moreover, since $c_{\lambda_k}\ge c_\lambda$, we have
$$I_\lambda(u_\lambda)=\lim_{k\to\infty}I_{\lambda_k}(u_k)\geq c_\lambda>0,$$
hence showing that $u_\lambda\not\equiv 0$, as wanted.
\hfill $\square$

\section{Non-existence for small $\lambda$}\label{sec5}

We also have a non-existence result for $\lambda$ small enough, analogous to \cite[Thm. 5.10]{ST}.

\begin{trm}
There exists a constant $\Lambda_0>0$ such that for $\lambda\in[0,\Lambda_0[$, $u\equiv 0$ is the only solution to $\eqref{eq3}$ in $E$.
\end{trm}

\begin{proof}
The Green function for $(-\Delta_g)^m$ is of the form 
$$
G_y(x)=\frac{2}{\Lambda_1}\log\frac{1}{d_g(x,y)} + \gamma(x,y),
$$
where $\gamma$ is smooth on $M\times M$.
If $u\in E$ solves \eqref{eq3}, then
\begin{equation}\label{8}
\begin{split}
u(y)&=\int_M(-\Delta_g)^mu G_y d\mu_g=
\lambda\frac{\int_Me^{2mu}G_yd\mu_g}{\int_Me^{2mu}d\mu_g}
\\
&\leq
\lambda\|\gamma\|_{L^\infty}
+
\frac{2\lambda}{\Lambda_1}\frac{\int_M\log\left(\frac{1}{d_g(y,x)}\right)e^{2mu(x)}d\mu_g(x)}{\int_Me^{2mu}d\mu_g}.
\end{split}
\end{equation}
We now use the inequality $ab\leq e^a+b(\log b-1)$ which holds for $b\in \R{}, a\in\R{+}$, and which follows from the identity
$$\sup_{a\in\R{}}\{ab-e^a\}=b(\log b-1),$$
choosing $a=-\log (d_g(y,\cdot))$, $b=e^{2mu}$, hence getting
$$
\log\left(\frac{1}{d_g(y,\cdot)}\right)e^{2mu}\leq \frac{1}{d_g(y,\cdot)} + 2mu e^{2mu}  -e^{2mu},
$$
and recalling that by the Jensen inequality $\int_Me^{2mu}d\mu\geq 1$, we infer
$$
\frac{\int_M\log\left(\frac{1}{d_g(y,\cdot)}\right)e^{2mu}d\mu_g}{\int_Me^{2mu}d\mu_g}\leq \frac{C}{2m-1} + \frac{2m\int_Me^{2mu} u d\mu_g}{\int_Me^{2mu} d\mu_g}.
$$
We now use \eqref{eq3} and notice that the above right-hand side does not depend on $y$ to show
\begin{equation}\label{supu}
\begin{split}
 \|u\|^2&=\lambda\frac{\int_Me^{2mu} u d\mu_g}{\int_Me^{2mu} d\mu_g}\leq \lambda \sup_M u \leq\frac{2\lambda^2C}{\Lambda_1(2m-1)}+\frac{4m\lambda}{\Lambda_1}\|u\|^2+\lambda^2\|\gamma\|_{L^\infty}\\
 &\le C\lambda^2+\frac{4m\lambda}{\Lambda_1}\|u\|^2.
\end{split}
\end{equation}
Then for $\lambda<\frac{\Lambda_1}{8m}$ we obtain
$$
\|u\|^2\leq\frac{C\Lambda_1\lambda^2}{\Lambda_1-4m\lambda}\leq C\lambda^2,
$$
and \eqref{supu} gives $\sup_M |u|\le C\lambda$ for $\lambda>0$ small enough. Therefore $|e^{2mu}-1|\leq e^{C\lambda}u$ and, recalling that $\int_M ud\mu_g=0$, we get
$$
\|u\|^2=\lambda\frac{\int_M(e^{2mu} -1)u d\mu_g}{\int_Me^{2mu} d\mu_g}\leq \lambda e^{C\lambda}\|u\|^2_{L^2}\leq C\lambda\|u\|^2.
$$
For $\lambda>0$ small enough this implies $\|u\|=0$, hence $u\equiv 0$, and this concludes the proof.
\end{proof}

\end{document}